\documentclass[12pt,twoside,reqno]{amsart}

\allowdisplaybreaks
\newtheorem{thm}{Theorem}[section]
\newtheorem{lemma}{Lemma}[section]
\newtheorem{rem}{Remark}[section]

\theoremstyle{definition}

\theoremstyle{remark}

\newcommand{\R}{{\mathbb R}}
\newcommand{\N}{{\mathbb N}}

\newcommand{\noi}{\noindent}

\numberwithin{equation}{section}

\begin{document}
\title[KS-ZK]	
{Multi-dimensional Kuramoto-Sivashinsky-Zakharov-Kuznetsov  equation posed on admisssible multi-dimensional domains}
\author{N. A. Larkin}	
\address
{
	Departamento de Matem\'atica, Universidade Estadual
	de Maring\'a, Av. Colombo 5790: Ag\^encia UEM, 87020-900, Maring\'a, PR, Brazil
}

\bigskip

\email{ nlarkine@uem.br;nlarkine@yahoo.com.br }
\date{}

\begin{abstract} An initial- boundary value problem for the n-dimensional ($n$ is a natural number from the interval [2,7]) Kuramoto-Sivashinsky-Zakharov-Kuznetsov equation posed on  smooth bounded  domains in $\R^n$  was considered. The existence and uniqueness of  global  regular  solutions  as well as their exponential decay  have been established. A connection between an order of the stationary part of the equation and  admissible dimensions of a domain has been revealed.
\end{abstract}

\maketitle

\section{Introduction}\label{introduction}
In this work, we study the existence, uniqueness, regularity  and  exponential decay  of global solutions to an initial-boundary value problem for the $n$-dimensional Kuramoto-Sivashinsky-Zakharov-Kuznetsov  equation (KS-ZK) 
\begin{align}
	& \phi_t+\Delta^2 \phi+ \Delta \phi+\Delta \phi_{x_1} +\frac{1}{2}|\nabla \phi|^2=0
\end{align}
with the purpose to reveal a connection between an order of the stationary part of the KS-ZK equation and admiissible dimensions of domains involved.
Here  $n$ is a natural number from the interval [2,7], $\Delta$ and $\nabla$ are the Laplacian and the gradient in $\R^n.$
In \cite{kuramoto}, Kuramoto studied the turbulent phase waves  and Sivashinsky in \cite{sivash} obtained an asymptotic equation which modeled the evolution of a disturbed plane flame front. See also \cite{Cross}.
Mathematical  results on  initial and initial-boundary value problems for one-dimensional (1.1)  are presented in \cite{benney,Iorio, Guo,cousin,feng,Larkin,temam1,temam2,zhang}, see references  there for more information. Multi-dimensional problems for equations of (1.1) type can be found in \cite{kukavica,gramchev,Guo,massat,molinet,temam1,sell,temam2}, where some results on existence, regularity and nonlinear stability of solutions have been presented. In one-dimensional case,  Kuramoto-Sivashinsky type equations which included the (KdV) term $\phi_{xxx}$ have been considered \cite{benney,Iorio,feng, Larkin, zhang}. The aim of these studies were intentions to understand simultaneous influences of dispersion and dissipation on stability of  systems modeled by these equations. In \cite{Larkin}, an initial-boundary value problem for the KdV equation was regularized by initial-boundary value problems for the Kuramoto-Sivashinsky equation with a consequent passage to the limit with respect to regularization parameters.\par In multi-dimensional cases, $\phi_{xxx}$ must be replaced by $\Delta \phi_{x_1}$, Zakharov-Kuznetsov term, see \cite{zk}, where the propagation of nonlinear ionic-sonic waves in a plasma submitted to a magnetic field directed along the $x_1$-axis was studied.This Kuramoto-Sivashinsky-Zakharov-Kuznetsov (KS-ZK)
equation may be useful in studies of nonlinear stability of a three-dimensional viscous film flowing down an inclined surface; see \cite{topper}, where it has been observed that the three-dimensional structure of waves can suppress an instability due to presence of the fourth-order derivative term.\par There is an extensive literature on existence, uniqueness and stability of global solutions for both ZK equation, see \cite{dor,fam,temam2} and for the Kuramoto-Sivashinsky equation, see \cite{Guo,feng,Larkin2,Larkin4,molinet,zhang}.  Recently, appeared published papers on solvability of initial-boundary value problems for two-dimensional and three-dimensional Kuramoto-Sivashinsky-Zakharov-Kuznetsov equation in bounded and unbounded domains, see \cite{Larkin1,Larkin3}. In \cite{Larkin2},  an initial-boundary value problem for the n-dimensional Kuramoto-Sivashinsky equation posed on smooth domains in $\R^n,\;n\in [2,7]$ has been studied\par This motivated us to consider n-dimensional Kuramoto-Sivashinsky-Zakharov-Kuznetsov equation and to define natural $n$ which allow to prove the existence, uniqueness and exponential decay of global regular solutions. Consequently, we call (1.1) Kuramoto-Sivashinsky-Zakharov-Kuznetsov equation (KS-ZK).\\
For $n$ dimensions, (1.1) can be rewritten in the form of the following system:

 \begin{align}
	&(u_j)_t+\Delta^2 u_j+\Delta u_j +(\Delta u_{j})_{x_1}+\frac{1}{2}\sum_{i=1}^n(u_i)^2_{x_j}=0,   \\
	&(u_i)_{x_j}=(u_j)_{x_i},\;i\ne j,\;i,j=1,...,n,
\end{align}
where $u_j=\phi_{x_j}.$ 
First essential problem that arises while one studies stability either for (1.1) or for (1.2)-(1.3) is a destabilizing effect of 
$\Delta u_j $ that may be damped by a dissipative term $\Delta^2 u_j$ provided a domain has some specific properties. Naturally, so called "thin domains"  appear  where some dimensions are small while the others may be  large, see \cite{Iftimie,sell}.\par   Our approach is based on the Faedo-Galerkin method and simple,  transparent conditions  for the geometry of a domain which suppress instability. These conditions were obtained using Steklov`s inequalities. \par Second essential problem is presence of semi-linear terms in (1.2) which are interconnected. Differently from the one-dimensional case, this does not allow to obtain the first estimate independent of time and solutions and leads to a connection between geometric properties of a domain and initial data.\\
Our work has the following structure: Section I is Introduction. Section 2 contains notations and auxiliary facts. In Section 3, formulation of an initial-boundary value problem for (1.2)-(1.3) posed on  bounded smooth domains in $\R^n$ with boundary conditions of a solution  and the Laplace operator of the solution  on  boundaries of domains are given. The existence and uniqueness of a global regular  solution, exponential decay of the $H^2$-norm  have been established.  Moreover, a "smoothing" effect has been observed  here.
 Section 4 contains conclusions.

\section{Notations and Auxiliary Facts}

let $D_n$ be a suffuciently smooth domain in $\R^n$ satisfying
the Cone condition, \cite{Adams}, and $x=(x_1,...,x_n) \in D_n,$ where $n$ is a fixed natural number $n\in [2,7].$ We use the standard notations of Sobolev spaces $W^{k,p}$, $L^p$ and $H^k$ for functions and the following notations for the norms \cite{Adams}
for scalar functions $f(x,t):$
$$\| f \|^2 = \int_{D_n} | f |^2dx, \hspace{1cm} \| f \|_{L^p(D_n)}^p = \int_{D_n} | f  |^p\, dx,$$
$$\| f \|_{W^{k,p}(D_n)}^p = \sum_{0 \leq \alpha \leq k} \|D^\alpha f \|_{L^p(D_n)}^p, \hspace{1cm} \| f \|_{H^k(D_n)} = \| f \|_{W^{k,2}(D_n)}.$$

When $p = 2$, $W^{k,p}(D_n) = H^k(D_n)$ is a Hilbert space with the scalar product 
$$((u,v))_{H^k(D_n)}=\sum_{|j|\leq k}(D^ju,D^jv),\;
\|u\|_{L^{\infty}(D_n)}=ess\; sup_{D_n}|u(x)|.$$
We use a notation $H_0^k(D_n)$ to represent the closure of $C_0^\infty(D_n)$, the set of all $C^\infty$ functions with compact support in $D_n$, with respect to the norm of $H^k(D_n)$.

\begin{lemma}[Steklov's Inequality \cite{steklov}] Let $v \in H^1_0(0,L).$ Then
	\begin{equation}\label{Estek}
	\frac {\pi^2}{L^2}\|v\|^2 \leq \|v_x\|^2.
	\end{equation}
\end{lemma}

\begin{lemma}
	[Differential form of the Gronwall Inequality]\label{gronwall} Let $I = [t_0,t_1]$. Suppose that functions $a,b:I\to \R$ are integrable and a function $a(t)$ may be of any sign. Let $u:I\to \R$ be a positive differentiable function satisfying
	\begin{equation}
		u_t (t) \leq a(t) u(t) + b(t),\text{ for }t \in I\text{ and } \,\, u(t_0) = u_0,
	\end{equation}
	then
	$$u(t) \leq u_0 e^{ \int_{t_0}^t a(\tau)\, d\tau } + \int^t_{t_0} e^{\int_{t_0}^s a(r) \, dr} b(s) ds$$\end{lemma}
\begin{proof} Multiply (2.2) by the integrating factor $e^{\int_{t_0}^t a(\tau)d\tau}$ and integrate the result from $t_0$ to $t.$
\end{proof} 
The next Lemmas will be used in  estimates:

\begin{lemma}[See: \cite{friedman}, Theorem 9.1]  Let  $n$ be a natural number from the interval $[2,7];  \;D_n$ be a sufficiently smooth bounded domain in $\R^n$ satisfying the Cone condition and $v \in H^4(D_n)\cap H^1_0(D_n)$.  then
	for all natural $n\in [2,7]$ 
	\begin{equation}
		\sup_{D_n}|v(x)|\leq C_n\| v\|_{H^4(D_n)}.
	\end{equation}
	The constant $C_n$ depends on $n, D_n.$	
\end{lemma}

\begin{lemma}\label{lemma3}
	Let $f(t)$ be a continuous  positive function such that
	
	\begin{align} & f'(t) + (\alpha - k f^n(t)) f(t) \leq 0,\;t>0,\;n\in\N,\label{lemao1}\\
		& \alpha - k f^n(0)> 0,\;k>0.\label{lemao2}\end{align}
	\noi Then
	
	\begin{equation}f(t) < f(0)\end{equation}
	
	\noi for all $t > 0$.
\end{lemma}

\proof{  Obviously, $f'(0) + (\alpha - k f^n(0)) f^n(0) \leq 0$. Since $f$ is continuous, there exists $T>0$ such that $f(t)<f(0)$ for every $t \in [0,T)$. Suppose that  $f(0) = f(T)$. Integrating \eqref{lemao1}, we find
	
	$$f(T)+ \int_0^T(\alpha - k f^n(t)) f(t) \, dt \leq f(0).$$
	
	\noi Since	$$  \int_0^T(\alpha - k f^n(t)) f(t) \, dt >0,$$ then $f(T)<f(0).$ This  contradicts that $f(T)=f(0).$ Therefore, $f(t) < f(0)$ for all $t > 0.$ \\
The proof of Lemma 2.4 is complete.  $\Box$

\section{ KS-ZK system posed on  bounded smooth domains in $\R^n$. Special basis }

 Let $\Omega_n$ be the minimal nD-parallelepiped containing a given bounded smooth domain $\Bar{D_n}\in\R^n,\;n=2,...,7$:
$$\Omega_n =\{ x\in \R^n; x_i\in (0,L_i)\},\; \;i=1,...,n.$$ Fixing any natural $n=2,...,7,$  consider  in $Q_n=D_n\times (0,t)$  the following initial-boundary value problem:

	\begin{align}
	(u_j)_t+\Delta^2 u_j+\Delta u_j +(\Delta u_{j})_{x_1}+\frac{1}{2}\sum_{i=1}^n(u_i)^2_{x_j}=0; &  \\
	(u_i)_{x_j}=(u_j)_{x_i},\;i\ne j;&\\
	u_j|_{\partial D_n}=\Delta  u_j|_{\partial D_n}=0,\; t>0;&\\
	u_j(x,0)=u_{j0}(x),\;j=1,...,n;\;x \in D_n.
\end{align}

 \begin{lemma}
	Let $f\in H^4(D_n)\cap H^1_0(D_n),\;\Delta f|_{\partial D_n}=0$.  Then
	\begin{align}
		&a\|f\|^2\leq\|\nabla f\|^2,\;\;a^2\|f\|^2\leq \|\Delta f\|^2,\;\;a\|\nabla f\|^2\leq \|\Delta f\|^2,\\
		& a^2\|\Delta f\|^2\leq \|\Delta^2 f\|^2,\;\;a\|\Delta f_{x_1}\|^2\leq \|\Delta^2 f\|^2,\;   \text{where} \; a=\sum_{i=1}^n\frac{\pi^2}{L_i^2}.
	\end{align}	
\end{lemma}

\begin{proof} By definition,
	$$\|\nabla f\|^2=\sum_{i=1}^n\|f_{x_i}\|^2.$$
Extending $f\in H^1_0(D_n)$ by $0$ into the minimal parallelepiped $\Omega_n$ as $\Bar{f}$, making use of Steklov`s inequalities in $\Omega_n$ and taking into account that $\|\nabla \Bar{f}\|_{\Omega_n}=\|\nabla f\|_{D_n},$  we get
	$$\|\nabla f\|^2\geq \sum_{i=1}^n\frac{\pi^2}{L_i^2}\|f\|^2=a\|f\|^2.$$
	On the other hand,
	$$a\|f\|^2\leq \|\nabla f\|^2 =-\int_{D_n}f\Delta fdx\leq \|\Delta f\|\|f\|.$$
	This implies
	$$a\|f\|\leq \|\Delta f\|\;\;\text{and} \;\;a^2\|f\|^2\leq \|\Delta f\|^2.$$
	Consequently, \;$a\|\nabla f\|^2\leq \|\Delta f\|^2.$  Similarly,
	$$\|\Delta f\|^2=(\Delta^2 f,f)\leq \|\Delta^2 f\|\|f\|\leq\frac{1}{a}\|\Delta^2 f\|\|\Delta f\|,$$
	$$\|\Delta f_{x_1}\|^2\leq \|\nabla\Delta f\|^2=-(\Delta^2f,\Delta f)\leq\|\Delta^2 f\|\|\Delta f\|\leq\frac{1}{a}\|\Delta^2 f\|^2.$$
	Assertions of Lemma 3.1 follow from these inequalities.
	 	Proof of Lemma 3.1 is complete.
\end{proof}
\begin{rem} Assertions of Lemma 3.1 are true if the function $f$ is replaced respectively by $u_j,\;j=1,...,n$.
\end{rem}

\begin{lemma} Let $a>1$. In conditiions of Lemma 3.1,
	\begin{align}\|f\|^2(t)_{H^2(D_n)}\leq 3\|\Delta f\|^2(t), &\\
		\|f\|^2(t)_{H^4(D_n)}\leq 5\|\Delta^2f\|^2(t), &\\
		\sup_{D_n}|f(x)|\leq C_s\| \Delta^2 f\|,\; \text{where}\;C_s=5C_n.
	\end{align}
\end{lemma}
\begin{proof} To prove (3.8), making use of Lemma 3.1, we find
	$$\|f\|^2_{H^4(D_n)}=\|f\|^2+\|\nabla f\|^2+\|\Delta f\|^2+\|\nabla\Delta f\|^2+\|\Delta^2 f\|^2$$$$
	\leq\Big(\frac{1}{a^4}+	\frac{1}{a^3}+\frac{1}{a^2}+\frac{1}{a}+1\Big)\|\Delta^2 f\|^2.$$
	Since $a>1$, then  (3.8) follows. Similarly, (3.7) can be proved. Moreover, taking into account Lemma 2.3, we get (3.9).
\end{proof}

\begin{thm} Let 
		Let  $n$ be a natural number from the closed interval [2,7]; $D_n\in\R^n$ be a bounded smooth domain satisfying the Cone condition and $\Omega_n$ be the minimal $nD$-parallelepiped containing $\Bar{D_n}.$ Let
	\begin{align} 2a=2\sum_{i=1}^n\frac{\pi^2}{L_i^2}>3+5^{1/2},\;\;\theta=1-\frac{1}{a}-\frac{1}{a^{1/2}}>0.
	\end{align}
	Given $$u_{j0}(D_n)\in H^2(D_n)\cap H^1_0(D_n),\;j=1,...,n$$ such that
	
	\begin{align}
		\theta-\frac{2 C_s^2 7^3}{a\theta}\Big(\sum_{j=1}^n\|\Delta u_j\|^2(0)\Big)>0,
	\end{align}
	then there exists a unique global regular solution to (3.1)- (3.4):
	$$ u_j\in L^{\infty}(\R^+; H^2(D_n)\cap H^1_0(D_n))\cap L^2(\R^+;H^4(D_n)\cap H^1_0(D_n));$$$$u_{jt}\in L^2(\R^+;L^2(D_n)), \;j=1,...,n.$$ Moreover,
	\begin{align}\sum_{j=1}^n \|\Delta u_j\|^2(t)
		\leq \Big(\sum_{j=1}^n\|\Delta u_{j0}\|^2\Big)\exp\{-a^2t\theta/2\}
	\end{align}
	and
	\begin{align}
		\sum_{i=1}^n\|\Delta u_i\|^2(t)+\int_0^t\sum_{i=1}^n\|\Delta^2 u_i\|^2(\tau)d\tau \leq C\sum_{i=1}^n\|\Delta u_{i0}\|^2, \;t>0.\notag
	\end{align}
\end{thm}

\begin{proof}
	
	Let $\{w_i(x)$ be eigenfunctions of the following problem:
$$ \Delta^2 w_i -\lambda_i w_i=0,\; x\in D_n;\; w_i|_{\partial D_n}=\Delta w_i|_{\partial D_n}=0.$$ 
We  construct approximate solutions to (3.1)-(3.4) in the form
$$u^N_j(x,t)=\sum_{i=1}^N g_i^j(t)w_i(x);\;j=1,...,n.$$
Unknown functions $g_i^j(t)$\;\;satisfy the following initial problems:
\begin{align}\frac{d}{dt}(u^N_j,w_j)(t)+(\Delta u^N_j,\Delta w_j)(t)-(\nabla u^N_j, \nabla w_j)(t)&\notag\\
	-((\nabla u^N_j)_{x_1},\nabla w_{j})(t)	+\frac{1}{2}\sum_{i=1}^n((u^N_i)^2_{x_j},w_j)(t)=0,&\\
	g_i^j(0)=g_{i0}^j, \;\j=1,...,n .
\end{align}
By Caratheodory`s existence theorem, \cite{carath}, Theorems 1.2 of Chapter 1,\; there exist solutions to (3.13)-(3.14) at least locally in $t$.
All the estimates we will prove will be done on smooth solutions to (3.1)-(3.4). Naturally, the same estimates are true also for approximate solutions $u^N_j.$\\

\noi {\bf Estimate I }Multiply (3.1) by $2\Delta^2 u_j$  to obtain

\begin{align*}\frac {d}{dt}\|\Delta u_j\|^2(t)+2\|\Delta^2 u_j\|^2(t)-2\|\Delta^2 u_j\|(t)\|\Delta u_j\|(t)&\notag\\
	-2\|(\Delta u_j)_{x_1}\|(t)\|\Delta^2 u_j)\|(t)&\notag\\	+2\sum_{i=1}^n((u_i)(u_i)_{x_j},\Delta^2 u_j)(t)\leq0,\;j=1,...,n.
\end{align*}

By Lemma 3.1, $a\|\Delta u_j\|\leq \|\Delta^2 u_j\|,\;a^{1/2}\|(\Delta u_j)_{x_1}\|\leq \|\Delta^2 u_j\|,$ hence the last inequality can be rewritten as
\begin{align*}\frac {d}{dt}\|\Delta u_j\|^2(t)+2\Big(1-\frac{1}{a}-\frac{1}{a^{1/2}}\Big)\|\Delta^2 u_j\|^2(t)&\notag\\
	+2\sum_{i=1}^n((u_i)(u_i)_{x_j},\Delta^2 u_j)(t)\leq 0,\;j=1,...,n.
\end{align*}
By definition, $\theta=1-\frac{1}{a}-\frac{1}{a^{1/2}}$ must be positive. Solving the inequality
	$$1-\frac{1}{a}-\frac{1}{a^{1/2}}>0$$
		with respect to $a$, we find that $2a>3+5^{1/2}$, see (3.10). In \cite{Larkin1,Larkin3}, it was defined $a>3$. This condition is more restrictive, but more explicit. 

Making use of Lemma 3.1 and definition of $\theta$, we get

\begin{align}\frac {d}{dt}\|\Delta u_j\|^2(t)+2\theta\|\Delta^2 u_j\|^2(t) &\notag\\
	\leq	2\sum_{i=1}^n(\sup_D|u_i|(t)\|\nabla u_i\|(t))\|\Delta^2 u_j\|(t)
	,\;j=1,...,n.		
\end{align}
Taking into account Lemmas 2.5, 3.1, 3.2, we can rewrite (3.15) as

\begin{align}
	\frac{d}{dt} \|\Delta u_j\|^2(t)+2\theta\|\Delta^2 u_j\|^2(t)&\notag\\
	\leq 2\Big[\sum_{i=1}^n \sup_{D_n}| u_i(x,t)|\|\nabla u_i\|(t)\Big]\|\Delta^2 u_j\|(t)&\notag\\
	\leq
	2\Big[C_s\sum_{i=1}^n\|\Delta^2 u_i\|(t)\|\nabla u_i\|(t)\Big]\|\Delta^2 u_j\|(t); j=1,...,n.
\end{align}
Summing over $j=1,...,n$ and making use of Lemma 3.1, we  rewrite (3.16) in the form:

\begin{align*}
	\frac {d}{dt}\sum_{j=1}^n\|\Delta u_j\|^2(t)+2\theta \sum_{j=1}^n \|\Delta^2 u_j\|^2(t)&\notag\\
	\leq 2C_s n\Big(\sum_{j=1}^n\|\nabla  u_j\|(t)\Big)\Big[\sum_{j=1}^n \|\Delta^2 u_j\|^2(t)\Big]&\notag\\
	\leq \Big[\frac{\theta}{2}+\frac{2C_s^2 n^2}{\theta}\Big(\sum_{j=1}^n\|\nabla u_j\|(t)\Big)^2\Big]\sum_{j=1}^n\|\Delta^2 u_j\|^2(t)&\notag\\
	\leq \Big[\frac{\theta}{2}+\frac{2C_s^2 n^3}{\theta}\Big(\sum_{j=1}^n\|\nabla u_j\|^2(t)\Big)\Big]\sum_{j=1}^n\|\Delta^2 u_j\|^2(t)&\notag\\
	\leq \Big[\frac{\theta}{2}+\frac{2C_s^2 n^3}{a\theta}\Big(\sum_{j=1}^n\|\Delta u_j\|^2(t)\Big)\Big]\sum_{j=1}^n\|\Delta^2 u_j\|^2(t).
\end{align*}
Taking this into account, we transform (3.16) as follows:

\begin{align}
	\frac {d}{dt}\sum_{j=1}^n\|\Delta u_j\|^2(t)+\frac{\theta}{2} \sum_{j=1}^n \|\Delta^2 u_j\|^2(t)&\notag\\
	+ \Big[\theta-\frac{2C_s^2 n^3}{a\theta}\Big(\sum_{j=1}^n\|\Delta u_j\|^2(t)\Big)\Big]\sum_{j=1}^n\|\Delta^2 u_j\|^2(t)\leq 0.
\end{align}

Condition (3.11) and Lemma 2.4 guarantee that

$$\theta-\frac{2C_s^2 n^3}{a\theta}\Big(\sum_{j=1}^n\|\Delta u_j\|^2(t)\Big)>0,\; \;t>0.$$

Hence, by Lemma 3.1, (3.17) can be rewritten as
\begin{align}
	\frac {d}{dt}\sum_{j=1}^n\|\Delta u_j\|^2(t)+\frac{a^2\theta}{2}\sum_{j=1}^n \|\Delta u_j\|^2(t)\leq 0.
\end{align}
Integrating, we get

\begin{align}\sum_{i=1}^n \|\Delta u_j\|^2(t)	\leq\sum_{j=1}^n\|\Delta u_{j0}\|^2\exp\{-a^2\theta t/2\}.
\end{align}
Since
$$\theta-\frac{2C_s^2 n^3}{a\theta}\Big(\sum_{j=1}^n\|\Delta u_j\|^2(t)\Big)>0,\; \;t>0,$$

returning to (3.17), we find
\begin{align}
	\sum_{i=1}^n\|\Delta u_i\|^2(t)+\int_0^t\sum_{i=1}^n\|\Delta^2 u_i\|^2(\tau)d\tau \leq C\sum_{i=1}^n\|\Delta u_{i0}\|^2.
\end{align}
By Lemma 3.1, $a^{1/2}\|\nabla \Delta u_j\|\leq \|\Delta^2 u_j\|,$ then
 directly from (3.1),  $$(u_j)_t \in L^2(\R^+;L^2(D_n)),\;j=1,...,n.$$
 Obviously, these inclusions and estimates (3.19), (3.20) do not depend on $N$ that allow us to pass to the limit as $n\to \infty$ in (3.13), (3.14) and to prove the existence part of Theorem 3.1.

\begin{lemma} A regular solution of Theorem 3.1 is uniquely defined.
\end{lemma}
\begin{proof}\begin{proof} Let $u_j$ and $v_j$,\;$j=1,...,n$ be two distinct solutions to (3.1)-(3.4). Denoting $w_j=u_j-v_j$,	we come to the following system:
		\begin{align}
			(w_j)_t+\Delta^2 w_j+\Delta w_j +\Delta(u_j)_{x_1} +\frac{1}{2}\sum_{i=1}^n\Big(u^2_i-v^2_i\Big)_{x_j} =0,&\\	
			(w_j)_{x_i}=(w_i)_{x_j},\;i\ne j,&\\
			w_j|_{\partial D_n}=\Delta w_j|_{\partial D_n}
			=0,\; t>0;&\\
			w_j(x,0)=0 \;\text{in} \;D_n,\;\;j=1,...,n.
		\end{align}
		Multiply (3.21) by $2w_j$ to obtain
		
		\begin{align} \frac{d}{dt}\|w_j\|^2(t)+2\|\Delta w_j\|^2(t)-2\|\nabla w_j\|^2(t)&\notag\\
			-2\|\nabla (w_i)_{x_1}\|(t)\|\nabla w_j\|(t)\leq\sum_{i=1}^n(|\{u_i+v_i\}w_i,(w_j)_{x_j}|)(t).
		\end{align}
		
		Making use of Lemmas 2.3, 3.1, 3.2, we estimate
		$$I_1=2\|\nabla (w_i)_{x_1}\|(t)\|\nabla w_j\|(t)\leq 2\frac{1}{a^{1/2}}\|\Delta w_j\|^2(t),$$$$2\|w_j\|^2(t)\leq
		\frac{2}{a}\|\Delta w_j\|^2(t).$$
$$		I_2=\sum_{i=1}^n(|\{u_i+v_i\}w_i,(w_j)_{x_j}|) \leq \sum_{i=1}^n(\sup_{D_n}|\{u_i+v_i\}|\|w_i\|\|(w_j)_{x_j}\|)$$
		$$\leq C_s\sum_{i=1}^n\|w_i\|\Big(\|\Delta^2 (u_i+v_i)\|\Big)\|\nabla (w_j)\|$$
		$$\leq \frac{C_s}{a^{1/2}}\sum_{i=1}^n\|w_i\|\Big(\|\Delta^2 (u_i+v_i)\|\Big)\|\Delta w_j\|$$
				$$\leq \frac{\epsilon}{2}\|\Delta w_j\|^2	+\frac{nC_s^2}{2a\epsilon}\Big[\sum_{i=1}^n\{\|\Delta^2 u_i\|^2+\|\Delta^2 v_i\|^2\}\Big]\sum_{j=1}^n\|w_j\|^2.$$ Here $\epsilon$ is an arbitrary positive number. Substituting $I_1, I_2$ into (3.25), we get

		\begin{align}
			\frac{d}{dt}\|w_j\|^2(t)+(2\theta-\frac{\epsilon}{2})\|\Delta w_j\|^2(t)&\notag\\
			\leq \frac{nC_s^2}{2a\epsilon}\Big[\sum_{i=1}^n\{\|\Delta^2 u_i\|^2+\|\Delta^2 v_i\|^2\}\Big]\sum_{j=1}^n\|w_j\|^2.
		\end{align}
		Taking $\epsilon=4\theta$ and summing up over $j=1,...,n,$  we transform (3.26) as follows:
		\begin{align*}\frac{d}{dt}\sum_{j=1}^n\|w_j\|^2(t)\leq C \Big[\sum_{i=1}^n\{\|\Delta^2 u_i\|^2+\|\Delta^2 v_i\|^2\}\Big]\sum_{j=1}^n\|w_j\|^2,\;i=1,..., n.
		\end{align*}
		
		Since by (3.20),$$\|\Delta^2 u_i\|^2(t), \|\Delta^2 v_i \|^2(t)\in L^1(\R^+),$$
		 then by  Lemma 2.2, 

		$$ \|w_j\|^2(t)\equiv 0\;\;j=1,...,n,\; \text{for all}\;\; t>0.$$
		Hence
		\begin{align*}
			u_j(x,t)\equiv v_j(x,t);j=1,...,n.
		\end{align*}
		The proof of Lemma 3.3, and consequently, Theorem 3.1 are complete.
		
	\end{proof}
	
\end{proof} 
\end{proof}	
{\Large }

\section{ Conclusions}

In this work,  we studied the initial-boundary value problem for the n-dimensional Kuramoto-Sivashinsky-Zakharov-Kuznetsov system (1.2), (1.3) posed on bounded smooth domains $D_n\in\R^n.$. The problem on a bounded domains $D_n$  has homogeneous  conditions for a solution and the Laplace operator of the solution on the boundaries of the domains. Making use of a special base, we established the existence, uniqueness and exponential decay of the $H^2$-norm as well as "smoothing effect": for initial data from $H^2(D_n)\cap H^1_0(D_n)$, we  proved that  for all natural $n\in[2,7]$  solutions pertained to $L^2(\R^+;H^4(D_n)).$ A set of admissible domains  has been defined which eliminate  destabilizing effects of terms $\Delta u_j$ by dissipation of $\Delta^2 u_j.$  \\
 Since these problems do not admit the first  estimate independent of $t$ and solutions, in order to prove the existence of global  solutions, we put conditions  connecting geometrical properties of domains with initial data, see  (3.11).\par For $n=2,3,$ initial-boundary value problems have been studied in \cite{Larkin1,Larkin3}, nevertheless, we included these cases because they are covered by the same scheme and are special cases of the general results.

\section*{Conflict of Interests}

The author declares that there is no conflict of interest regarding the publication of this paper.

\medskip

\bibliographystyle{torresmo}

\begin{thebibliography}{99}
	
	
\bibitem{Adams}
\newblock R.A. Adams  and J.F. Fournier,
\newblock Sobolev Spaces.
\newblock Elsevier Science Ltd; Oxford OX5 IGB UK, @003o.

\bibitem{kukavica}
\newblock S. Benachour, I. Kukavica, W. Rusin, M. Ziane,
\newblock Anisotropic estimates for the two-dimensional Kuramoto-Sivashinsky equation,
\newblock J. of Dynamics and Differential Equations, Springer Verlag, 2014, 26, pp. 461-476.
          10.1007/s10884-014-9372-3. hal-00790207.
          
\bibitem{benney}
\newblock D.J. Benney,
\newblock Long waves on liquid films,
\newblock J.M.P., 45 (1966),150-155.
          

\bibitem{Iorio}
\newblock H.A. Biagioni, J.L. Bona, R.J. Iorio Jr., and M. Scialom,
\newblock On the Korteweg de Vries-Kuramoto-Sivashinsky Equation,
\newblock Adv. Diff. Eqs., 1(1) (1996),1-20.


\bibitem{gramchev}
\newblock H.A. Biagioni, I. Gramchev,
\newblock Multi-dimensional Kuramoto-Sivashinsky type equations: Singular initial data and analytic regularity,
\newblock Matemática Contemporanea, 15 (1998), 21-42.






\bibitem{carath}
\newblock E. Coddington, N. Levinson, 
\newblock Theory of Ordinary Differential Equations.
\newblock  New York: MacGraw-Hill; 1955.


\bibitem{Guo}
\newblock Guo Boling,
\newblock The existence and nonexistence of a global solution for the initial-value problem
          of generalized Kuramoto-Sivashinsky type equations,
\newblock J. of Mathematical Research and Exposition, vol. 11, N0 1,(1991), 57-69.

\bibitem{cousin}
\newblock A.T Cousin and N.A. Larkin,
\newblock Kuramoto-Sivashinsky equation in domains with moving boundaries,
\newblock Portugaliae Mathematica, vol. 59, fasc. 3 (2002), 335-349.

\bibitem{Cross}
\newblock M.C. Cross,
\newblock Pattern formation outside of equilibrium,
\newblock Review of Modern Physics. 65(3) (1993), 851-1086.


\bibitem{dor}
\newblock G.G. Doronin, N.A. Larkin,
\newblock Stabilization of regular solutions for the Zakharov-Kuznetsov equation posed on bounded rectangles and on a strip,
\newblock Proceeding of the Edinburgh Mathematical Society, 2(58), N3, 661--682 (2015)

\bibitem{fam}
\newblock A.V. Faminskii,
\newblock {Initial-boundary value problems in a half-strip for two-dimensional Zakharov-Kuznetsov equation,}
\newblock Ann. Inst. H. Poincaré (C) Analyse Non Linéaire 35, 1235-1265 (2018).

\bibitem{feng}
\newblock Feng Bao-Feng, B.A. Malomed, T. Kawahara,
\newblock Stable Periodic Waves in Coupled Sivashinsky-Korteweg-de Vries Equations,
\newblock arXiv:nlin/0209003 v 1 [nlin. PS] 1Sep 2002.


\bibitem{friedman}
\newblock A. Friedman,
\newblock Partial differential equations,
\newblock Dover publication, Mineola, N.Y. 1997.



\bibitem{Iftimie}
\newblock  D. Iftimie, G. Raugel,
\newblock Some Results on the Navier-Stokes Equations in Thin 3D Domains,
\newblock Journal of Differential Equations 169, (2001), 281-331.

\bibitem{kuramoto}
\newblock Y. Kuramoto and T. Tsuzuki,
\newblock On the formation of dissipative structures in reaction-diffusion systems,
\newblock Progr. Theor. Phys., 54 (1975), 687-699.

\bibitem{Larkin}
\newblock N.A. Larkin,
\newblock Korteweg de Vries and Kuramoto-Sivashinsky equations in bounded domains.
\newblock J. Math. Anal. Appl., 297 (2004), 169-185. 


\bibitem{Larkin1}
\newblock N.A. Larkin, 
\newblock Existence and decay of global solutions for the three-dimensional Kuramoto-Sivashinsky-Zakharov-Kuznetsov equation ,
\newblock JMAA online, doi:org/101016/jmaa.2022.126046



\bibitem{Larkin2}
\newblock N.A. Larkin,
\newblock Global regular solutions for the multi-dimensional Kuramoto-Sivashinsky equation posed on smooth domains.
\newblock arXiv:2203.14739v1 [math.AP] 28 Mar 2022 


\bibitem{Larkin3}
\newblock N.A. Larkin, 
\newblock Existence and decay of global solutions for the Kuramoto-Sivashinsky-Zakharov-Kuznetsov equation posed on rectangles,
\newblock PDEA, doi: 10.1007/s42.985-022-00155-6


\bibitem{Larkin4}
\newblock N.A. Larkin, 
\newblock Regularity  and decay of  solutions for the 3D Kuramoto-Sivashinsky equation posed on smooth domains and parallelepipeds,
\newblock Electron J. Math., 3 (2022), 1-15; doi:10.47443/ejm.2022.002

 \bibitem{massat} 
\newblock D. Massatt,
\newblock On the well-posedness of the anisotropically-reduced two-dimensional Kuramoto-Sivashinsky equation. 
\newblock DCDS, serie B, https://doi.org/10.3934/dcdsb.2021305
 
\bibitem{molinet} 
\newblock L. Molinet,
\newblock Local dissipativity in $L^2$ for the Kuramoto-Sivashinsky equation in spatial dimension 2, 
\newblock J. Dynam. Diff. Eq., 12 (2000), N0 3, 533-556.


\bibitem{temam1}
\newblock B. Nicolaenko, B. Scheurer and R. Temam,
\newblock Some global dynamical properties of the Kuramoto-Sivashinsky equations: nonlinear stability and attractors,
\newblock Phys. D 16 (1985), No 3, 155-183




\bibitem{sell} 
\newblock G.R. Sell and M. Taboada,
\newblock Local dissipativity and attractors for the Kuramoto-Sivashinsky equation in thin 2D domains,
\newblock Nonlin. Anal. 18 (1992), 671-687.


	
\bibitem{sivash}
\newblock G.I. Sivashinsky,
\newblock Nonlinear analysis of hydrodynamic instability in laminar flames-1. Derivation of basic equations,
\newblock Acta Astronautica, 4 (1977), 1177-1206
		


\bibitem{temam2}
\newblock J.-S. Saut, R. Temam, C. Wang,
\newblock {An initial and boundary-value problem for the Zakharov-Kuznetsov equation in a bounded domain,}
\newblock J. Math. Phys.,  53, 115612 (2012) doi:10.1063/1.4752102
		
	
\bibitem{steklov}
\newblock A.V. Steklov,
\newblock The problem of cooling of an heterogeneous rigid rod,
\newblock Communs. Kharkov Math. Soc., Ser. 2, 5 (1896) 136-181 (Russian).


	
\bibitem{temam3}
\newblock R. Temam,
\newblock Infinite-Dimensional Dynamical Systems in Mechanics and Physics,
\newblock  Springer, Berlin-Heidelberg, New York (1988).

\bibitem{topper}
\newblock J. Topper and T. Kawahara,
\newblock Approximate Equations for Long Nonlinear Waves on a Viscous Fluid,
\newblock J.Phys. Soc. of Japan,44 No 2,(1978), 663-666.

\bibitem{zk}
\newblock V. E. Zakharov and E. A. Kuznetsov,
\newblock On three-dimensional solitons,
\newblock Sov. Phys. JETP 39 (1974), 285--286.


  
\bibitem{zhang}
\newblock Jing Li, Bing-Yu Zhang and Zhixiong Zhang,
\newblock A non-homogeneous boundary value problem for the Kuramoto-Sivashinsky equation in a quarter plane,
\newblock  arXiv:1607.00506v2 [math.AP] 14 Jul 2016.
\end{thebibliography}

\end{document}